\documentclass[a4paper, 12pt, leqno]{amsart}
\usepackage{amsmath}
\usepackage{amscd}
\usepackage{amssymb}
\usepackage{amsthm}
\usepackage{amsfonts}
\usepackage{latexsym}

\usepackage[dvipdfmx]{graphicx, xcolor}

\theoremstyle{plain}
\newtheorem{lemma}{{\sc Lemma}}[section]

\newtheorem{proposition}[lemma]{{\sc Proposition}}
\newtheorem{theorem}[lemma]{{\sc Theorem}}

\theoremstyle{definition}
\newtheorem{remark}[lemma]{{\sc Remark}}

\numberwithin{equation}{section}

\def\Gh{{\mathfrak{h}}}

\def\BF{{\mathbb{F}}}

\def\BQ{{\mathbb{Q}}}
\def\BZ{{\mathbb{Z}}}

\def\Ker{\mathop{\rm Ker\hskip.5pt}\nolimits}

\begin{document}
\title[The Drinfeld pairing of a quantum group]
{Invariance of the Drinfeld pairing\\ of a quantum group}
\author{Toshiyuki TANISAKI}
\dedicatory{dedicated to Ken-ichi Shinoda in friendship and respect}
\thanks
{
The author was partially supported by Grants-in-Aid for Scientific Research (C) 24540026 
from Japan Society for the Promotion of Science.
}
\address{
Department of Mathematics, Osaka City University, 3-3-138, Sugimoto, Sumiyoshi-ku, Osaka, 558-8585 Japan}
\email{tanisaki@sci.osaka-cu.ac.jp}
\date{}
\begin{abstract}
We give two alternative proofs of the invariance of  the Drinfeld pairing under the action of the braid group.
One uses the Shapovalov form, and the other uses a characterization of the universal $R$-matrix.
\end{abstract}
\maketitle

\section{Introduction}
Let $U$ be the quantized enveloping algebra over $\BQ(q)$ associated to a symmetrizable generalized Cartan matrix $(a_{ij})_{i,j\in I}$.
We have the triangular decomposition $U=U^-U^0U^+$, where 
$U^0$ is the Cartan part, $U^+=\langle e_i\mid i\in I\rangle$ is the positive part, and $U^-=\langle f_i\mid i\in I\rangle$ is the negative part of $U$ respectively.
In application of the theory of quantized enveloping algebras to other fields such as mathematical physics and knot theory, the universal $R$-matrix plays a crucial role.
For example, to each representation of the quantized enveloping algebra one can construct a knot invariant by specializing the universal $R$-matrix.
Therefore, it is an important problem to give an explicit description of the universal $R$-matrix.
This problem is equivalent to giving an explicit description of 
the Drinfeld pairing $\tau:U^+\times U^-\to\BQ(q)$, which is  a bilinear form characterized by certain properties,
since the universal $R$-matrix is defined in terms of $\tau$ (see \cite{D}).
On the other hand, the Drinfeld pairing $\tau$ plays an important role in many aspects of the representation theory.
For example, in the finite case various properties of representations when $q$ is not a root of 1 are deduced using properties of $\tau$ (see for example \cite{Jan}).

For $i\in I$, denote by $T_i:U\to U$ the algebra automorphism introduced by Lusztig \cite{Lbook}
(in the finite case there is a different definition due to Levendorskii and Soibelman \cite{LS}).
It is a lift of the simple reflection of the Weyl group.
Let $W=\langle s_i\mid i\in I\rangle$ be the Weyl group.
Let $w\in W$, and
take a reduced expression  $w=s_{i_1}\cdots s_{i_r}$.
Set
\[
e_{\beta_k}=T_{i_1}\cdots T_{i_{k-1}}(e_{i_k}),
\qquad
f_{\beta_k}=T_{i_1}\cdots T_{i_{k-1}}(f_{i_k})
\]
for $k=1,\dots, r$.
Then there is a well-known formula for the value of
\begin{equation}
\label{eq:tef}
\tau(e_{\beta_r}^{m_r}\cdots e_{\beta_1}^{m_1},
f_{\beta_r}^{n_r}\cdots f_{\beta_1}^{n_1})
\end{equation}
(see \cite{KR}, \cite{LS}, \cite{Lbook}).
In the finite case 
$\{e_{\beta_r}^{m_r}\cdots e_{\beta_1}^{m_1}\mid m_i\geqq0\}$
(resp. $\{f_{\beta_r}^{n_r}\cdots f_{\beta_1}^{n_1}\mid n_i\geqq0\}$) forms a basis of 
$U^+$ (resp. $U^-$) so that
the formula for \eqref{eq:tef} gives an explicit description of $\tau$.
A crucial step in the proof of the formula for \eqref{eq:tef} (using Lusztig's definition of $T_i$) is the following invariance property;
\begin{align}
\label{eq:tef2}
&\tau(T_i^{-1}(x),T_i^{-1}(y))=\tau(x,y)
\\
\nonumber
&\qquad\qquad
(x\in U^+\cap T_i(U^+),\; y\in U^-\cap T_i(U^-)).
\end{align}
The original proof of this result in \cite{Lbook} uses lengthy computation concerning certain generating sets of the algebras $U^\pm\cap T_i(U^\pm)$ (in the detailed account \cite{Jan} it occupies whole Chapter 8A).

The aim of this note is to give two shorter proofs of \eqref{eq:tef2}.
The first one relies on a relation between the Drinfeld pairing and the Shapovalov form given in Proposition \ref{prop:DS} below.
We think Proposition \ref{prop:DS} is of independent interest.
The second one uses a well known characterization of the universal $R$-matrix.
We hope our investigation in this paper concerning $\tau$ including the new proofs of \eqref{eq:tef2} will be useful in the future especially in developing the representation theory of quantized enveloping algebras.

The second proof using the universal $R$-matrix has been obtained in examining the comment by H. Yamane suggesting the possibility of another approach along the line of Levendorskii and Soibelman \cite{LS}.
I would like to thank Hiroyuki Yamane for this crucial suggestion.

\section{quantized enveloping algebras}
Assume that we are given a finite-dimensional vector space $\Gh$ over $\BQ$ and 
linearly independent subsets 
$\{h_i\}_{i\in I}$, $\{\alpha_i\}_{i\in I}$ of $\Gh$, $\Gh^*$ respectively such that 
$(\langle \alpha_j,h_i\rangle)_{i, j\in I}$ is a symmetrizable generalized Cartan matrix.
Set 
\[
Q=\sum_{i\in I}\BZ\alpha_{i},\qquad
Q^+=\sum_{i\in I}\BZ_{>0}\alpha_i.
\]
The Weyl group $W$ is the subgroup of $GL(\Gh)$ generated by the involutions 
$s_i$ ($i\in I$) defined  by $s_i(h)=h-\langle\alpha_i,h\rangle h_i$ for $h\in \Gh$.
The contragredient action of $W$ on $\Gh^*$ is given by $s_i(\lambda)=\lambda-\langle\lambda,h_i\rangle\alpha_i$ ($i\in I$, $\lambda\in\Gh^*$).
Set $E=\sum_{i\in I}\BQ\alpha_i$.
We can take a $W$-invariant symmetric bilinear form
\[
(\;,\;):E\times E\to\BQ
\]
such that 
$\frac{(\alpha_i,\alpha_i)}2\in\BZ_{>0}$ for any $i\in I$.
Then we have $(\alpha_i,\alpha_j)\in\BZ$ for $i, j\in I$.
We assume that we are given a $\BZ$-form $\Gh_\BZ$ of $\Gh$ such that
$\langle\alpha_i,\Gh_\BZ\rangle\subset\BZ$ and $t_i:=\frac{(\alpha_i,\alpha_i)}2h_i\in\Gh_\BZ$ for any $i\in I$.
For $\gamma=\sum_in_i\alpha_i\in Q$ set $t_\gamma=\sum_in_it_i$.
Then we have $\langle \gamma,t_\delta\rangle=(\gamma,\delta)$ for $\gamma, \delta\in Q$.

For $n\in\BZ_{\geqq0}$ set 
\[
[n]_x=\frac{x^n-x^{-n}}{x-x^{-1}}\in\BZ[x,x^{-1}],
\quad
[n]!_x=[n]_x[n-1]_x\cdots[1]_x\in\BZ[x,x^{-1}].
\]

The quantized enveloping algebra $U$ associated to $\Gh$, $\{h_i\}_{i\in I}$, $\{\alpha_i\}_{i\in I}$, $(\;,\;)$, $\Gh_\BZ$ is the associative algebra over $\BF=\BQ(q)$ generated by the elements $k_h$, $e_i$, $f_i$ ($h\in\Gh_\BZ$, $i\in I$) satisfying the relations
\begin{align}
&
k_0=1,\qquad
k_hk_{h'}=k_{h+h'}
&(h, h'\in\Gh_\BZ),
\\
&
k_he_ik_{-h}=q^{\langle\alpha_i,h\rangle}e_i
&(h\in\Gh_\BZ, i\in I),
\\
&
k_hf_ik_{-h}=q^{-\langle\alpha_i,h\rangle}f_i
&(h\in\Gh_\BZ, i\in I),
\\
&
e_if_j-f_je_i=
\delta_{ij}
\frac{k_i-k_i^{-1}}{q_i-q_i^{-1}}
&(i, j\in I),
\\
&
\sum_{r+s=1-\langle\alpha_j,h_i\rangle}
(-1)^re_i^{(r)}e_je_i^{(s)}=0
&(i, j\in I,\;i\ne j),
\\
&
\sum_{r+s=1-\langle\alpha_j,h_i\rangle}
(-1)^rf_i^{(r)}f_jf_i^{(s)}=0
&(i, j\in I,\;i\ne j),
\end{align}
where $k_i=k_{t_i}$, $q_i=q^{(\alpha_i,\alpha_i)/2}$ for $i\in I$, and $e_i^{(r)}=\frac1{[r]!_{q_i}}e_i^r$, $f_i^{(r)}=\frac1{[r]!_{q_i}}f_i^r$ for $i\in I$, $r\in\BZ_{\geqq0}$.
For $\gamma\in Q$ we set $k_\gamma=k_{t_\gamma}$.

The associative algebra $U$ is endowed with a structure of Hopf algebra by
\begin{align}
&
\Delta(k_h)=k_h\otimes k_h,
\\
\nonumber
&
\Delta(e_i)=e_i\otimes1+k_i\otimes e_i,\quad
\Delta(f_i)=f_i\otimes k_i^{-1}+1\otimes f_i
\\
&\varepsilon(k_h)=1,\quad
\varepsilon(e_i)=\varepsilon(f_i)=0,
\\
&
S(k_h)=k_h^{-1},\quad
S(e_i)=-k_i^{-1}e_i,\quad
S(f_i)=-f_ik_i
\end{align}
for $h\in\Gh_\BZ, i\in I$.
We will sometimes use Sweedler's notation for the coproduct;
\[
\Delta(u)=\sum_{(u)}u_{(0)}\otimes u_{(1)}
\qquad(u\in U),
\]
and the iterated coproduct;
\[
\Delta_m(u)=\sum_{(u)_m}u_{(0)}\otimes\cdots\otimes u_{(m)}
\qquad(u\in U).
\]

Define subalgebras $U^0$, $U^+$, $U^-$, $U^{\geqq0}$, $U^{\leqq0}$ of $U$ by
\begin{gather*}
U^0=\langle k_h\mid h\in\Gh_\BZ\rangle,\quad
U^+=\langle e_i\mid i\in I\rangle,\quad
U^-=\langle f_i\mid i\in I\rangle,
\\
U^{\geqq0}=
\langle k_h, e_i\mid h\in\Gh_\BZ, i\in I\rangle,\quad
U^{\leqq0}=
\langle k_h, f_i\mid h\in\Gh_\BZ, i\in I\rangle.
\end{gather*}
Then we have 
\[
U^0=\bigoplus_{h\in\Gh_\BZ}\BF k_h.
\]
For $\gamma\in Q$ set 
\[
U_\gamma=
\{u\in U\mid k_huk_h^{-1}=q^{\langle\gamma,h\rangle}u\;(h\in\Gh_\BZ)\},
\qquad
U^\pm_\gamma=U_\gamma\cap U^\pm.
\]
Then we have
\[
U^\pm=\bigoplus_{\gamma\in Q^+}U^{\pm}_{\pm\gamma}.
\]
It is known that the multiplication of $U$ induces isomorphisms
\begin{gather*}
U\cong U^+\otimes U^0\otimes U^-
\cong U^-\otimes U^0\otimes U^+,
\\
U^{\geqq0}\cong U^+\otimes U^0\cong U^0\otimes U^+,
\quad
U^{\leqq0}\cong U^-\otimes U^0\cong U^0\otimes U^-
\end{gather*}
of vector spaces.

We define an algebra automorphism
\begin{equation}
\Phi:U\otimes U\to U\otimes U
\end{equation}
by
\[
\Phi(u\otimes u')=q^{-(\gamma,\delta)}uk_{-\delta}\otimes u'k_{-\gamma}
\qquad
(\gamma, \delta\in Q,\; u\in U_{\gamma}, \; u'\in U_{\delta}).
\]

Set 
\[
P=\{\lambda\in\Gh^*\mid \langle\lambda,\Gh_\BZ\rangle\subset\BZ\},
\quad
P^+
=
\{\lambda\in P\mid \langle\lambda,h_i\rangle\in\BZ_{\geqq0}\;(i\in I)\}.
\]
For a (left) $U$-module $V$ and $\lambda\in P$
we set
\[
V_\lambda=
\{v\in V\mid k_hv=q^{\langle\lambda,h\rangle}v\;(h\in\Gh_\BZ)\}.
\]
A $U$-module $V$ is said to be integrable if $V=\bigoplus_{\lambda\in P}V_\lambda$ and for any $v\in V$ and $i\in I$ there exists some $N>0$ such that $e_i^{(n)}v=f_i^{(n)}v=0$ for $n\geqq N$.
For $\lambda\in P^+$ define $U$-modules $V_+(\lambda)$, $V_-(-\lambda)$  by
\begin{align*}
V_+(\lambda)=&U/
\left(
\sum_{h\in\Gh_\BZ}
U(k_h-q^{\langle\lambda,h\rangle})
+
\sum_{i\in I}Ue_i
+
\sum_{i\in I}Uf_i^{(\langle\lambda,h_i\rangle+1)})
\right),
\\
V_-(-\lambda)=&U/
\left(
\sum_{h\in\Gh_\BZ}
U(k_h-q^{-\langle\lambda,h\rangle})
+
\sum_{i\in I}Uf_i
+
\sum_{i\in I}Ue_i^{(\langle\lambda,h_i\rangle+1)})
\right).
\end{align*}
They are known to be irreducible integrable $U$-modules.
For $\lambda\in P^+$ we set $v_\lambda=\overline{1}\in V_+(\lambda)$, and $v_{-\lambda}=\overline{1}\in V_-(-\lambda)$.

For $U$-modules $V$, $V'$ we regard $V\otimes V'$ as a $U$-module via the coproduct $\Delta:U\to U\otimes U$.
If $V$ and $V'$ are integrable, then so is $V\otimes V'$.

The following result follows easily from the proof of \cite[Lemma 2.1]{Jan}.
\begin{proposition}
\label{prop:pos}
The following conditions on $u\in U$ are equivalent to each other:
\begin{itemize}
\item[(a)] $u\in U^{\geqq0}$ $($resp.\ $u\in U^{\leqq0})$,
\item[(b)] for any integrable $U$-module $V$ and for any $\lambda\in P^+$ we have
$u(V\otimes v_\lambda)\subset V\otimes v_\lambda$
$($resp.\ $u(V\otimes v_{-\lambda})\subset V\otimes v_{-\lambda})$,
\item[(c)] for any integrable $U$-module $V$ and for any $\lambda\in P^+$ we have
$u(v_\lambda\otimes V)\subset v_\lambda\otimes V$
$($resp.\
$u(v_{-\lambda}\otimes V)\subset v_{-\lambda}\otimes V)$.
\end{itemize}
\end{proposition}

\section{braid group action}
We set
\[
\exp_x(y)=\sum_{n=0}^\infty\frac{x^{n(n-1)/2}}{[n]!_x}y^n\in(\BQ(x))[[y]].
\]
Then we have $\exp_x(y)\exp_{x^{-1}}(-y)=1$.

For $i\in I$ and $t\in\BF^\times$ we set
\[
\sigma_i(t)
=\exp_{q_i}(tq_i^{-1}k_ie_i)
\exp_{q_i}(-t^{-1}f_i)
\exp_{q_i}(tq_ik_i^{-1}e_i)
\]
(see \cite{Sa}).
It is regarded as an invertible operator on a integrable $U$-module.
Moreover, for any integrable $U$-module $V$ and any $\lambda\in P$ we have $\sigma_i(t)V_\lambda=V_{s_i\lambda}$.
If we are give $t_i\in\BF^\times$ for each $i\in I$, then the family $\{\sigma_i(t_i)\}_{i\in I}$ satisfies the braid relations.
We have
\begin{align*}
&\sigma_i(t)
\\
=&
\exp_{q_i}(tq_i^{-n-1}k_i^{n+1}e_i)
\exp_{q_i}(-t^{-1}q_i^{-n}k_i^{-n}f_i)
\exp_{q_i}(tq_i^{-n+1}k_i^{n-1}e_i)
\\
=&
\exp_{q_i}(-t^{-1}q_i^{-n-1}k_i^{-n-1}f_i)
\exp_{q_i}(tq_i^{-n}k_i^{n}e_i)
\exp_{q_i}(-t^{-1}q_i^{-n+1}k_i^{-n+1}f_i)
\end{align*}
for any $n\in\BZ$.

For $i\in I$ we define operators $q_i^{\pm h_i(h_i+1)/2}$ and $q_i^{\pm h_i(h_i-1)/2}$ on a integrable $U$-module $V$ by
\[
q_i^{\pm h_i(h_i+1)/2}v=q_i^{\pm \lambda(h_i)(\lambda(h_i)+1)/2}v,
\quad
q_i^{\pm h_i(h_i-1)/2}v=q_i^{\pm \lambda(h_i)(\lambda(h_i)-1)/2}v
\]
for $\lambda\in P$, $v\in V_\lambda$.
Then in the notation of \cite{Lbook} we have
\begin{align*}
T'_{i,-1}=q_i^{-h_i(h_i+1)/2}\sigma_i(-1),\qquad
T''_{i,-1}=q_i^{-h_i(h_i-1)/2}\sigma_i(1),
\end{align*}
and $T'_{i,1}=(T''_{i,-1})^{-1}$, $T''_{i,1}=(T'_{i,-1})^{-1}$.
\begin{remark}
If we extend the base field $\BF=\BQ(q)$ to $\BQ(q^{1/4})$, we can write
\begin{align*}
\sigma_i(t)
=&
q_i^{h_i^2/2}\exp_{q_i}(te_i)
q_i^{-h_i^2/4}\exp_{q_i}(-t^{-1}f_i)q_i^{-h_i^2/4}
\exp_{q_i}(te_i)
\\
=&
q_i^{h_i^2/2}\exp_{q_i}(-t^{-1}f_i)
q_i^{-h_i^2/4}\exp_{q_i}(te_i)q_i^{-h_i^2/4}
\exp_{q_i}(-t^{-1}f_i)
\\
=&
\exp_{q_i}(te_i)
q_i^{-h_i^2/4}\exp_{q_i}(-t^{-1}f_i)q_i^{-h_i^2/4}
\exp_{q_i}(te_i)q_i^{h_i^2/2}
\\
=&
\exp_{q_i}(-t^{-1}f_i)
q_i^{-h_i^2/4}\exp_{q_i}(te_i)q_i^{-h_i^2/4}
\exp_{q_i}(-t^{-1}f_i)q_i^{h_i^2/2}.
\end{align*}
\end{remark}

In the following we set $T_i=\sigma_i(-1)^{-1}q_i^{h_i(h_i+1)/2}$.
In the notation of \cite{Lbook} we have $T_i=T''_{i,1}$.
There exists a unique algebra automorphism $T_i:U\to U$ such that 
for any integrable $U$-module $V$ we have
$
T_iuv=T_i(u)T_iv$\;($u\in U, v\in V$).
Then we have $T_i(U_\gamma)=U_{s_i\gamma}$ for $\gamma\in Q$.
The action of $T_i$ on $U$ is given by
\begin{gather*}
T_i(k_h)=k_{s_ih},
\quad
T_i(e_i)=-f_ik_i,
\quad
T_i(f_i)=-k_i^{-1}e_i\qquad(h\in\Gh_\BZ),
\\
T_i(e_j)=
\sum_{r+s=-\langle \alpha_j,h_i\rangle}
(-1)^rq_i^{-r}e_i^{(s)}e_je_i^{(r)}\qquad(j\in I, i\ne j),
\\
T_i(f_j)=
\sum_{r+s=-\langle \alpha_j,h_i\rangle}
(-1)^rq_i^{r}f_i^{(r)}f_jf_i^{(s)}\qquad(j\in I, i\ne j)
\end{gather*}
(see \cite{Lbook}).
We can easily check that
\begin{equation}
\label{eq:Phi}
\Phi\cdot(T_i\otimes T_i)=(T_i\otimes T_i)\cdot\Phi.
\end{equation}

For $i\in I$ and integrable $U$-modules $V$, $V'$ we define operators $Z_i:V\otimes V'\to V\otimes V'$ and $R_i:V\otimes V'\to V\otimes V'$ by
\[
Z_i=
\exp_{q_i}((q_i-q_i^{-1})f_i\otimes e_i),
\qquad
R_i
=
\exp_{q_i}^{-1}(-(q_i-q_i^{-1})e_i\otimes f_i).
\]
They are invertible with
\[
Z_i^{-1}=P(R_i),
\]
where $P(x\otimes y)=y\otimes x$.

The following result is well-known (see \cite{KR}, \cite{LS}, \cite{Lbook}).
\begin{proposition}
\label{prop:T}
Let $V$ and $V'$ be integrable $U$-modules.
Then as an operator on $V\otimes V'$ we have
\begin{align*}
T_i=(T_i\otimes T_i)\cdot Z_i=\Phi^{-1}(R_i^{-1})\cdot (T_i\otimes T_i).
\end{align*}
\end{proposition}
\begin{lemma}
For $u\in U$ we have
\begin{align}
\label{eq:T1}
\Delta(T_i^{-1}(u))
=&
Z_i^{-1}\cdot(T_i^{-1}\otimes T_i^{-1})(\Delta(u))\cdot Z_i,
\\
\label{eq:T2}
\Delta(T_i(u))
=&
\Phi^{-1}(R_i^{-1})\cdot(T_i\otimes T_i)(\Delta(u))\cdot \Phi^{-1}(R_i).\end{align}
as operators  on the tensor product of two integrable $U$-modules.
\end{lemma}
\begin{proof}
By Proposition \ref{prop:T} we have
\begin{align*}
\Delta(T_i^{-1}(u))
=&T_i^{-1}\cdot\Delta(u)\cdot T_i
\\
=&
Z_i^{-1}\cdot(T_i^{-1}\otimes T_i^{-1})\cdot\Delta(u)\cdot(T_i\otimes T_i)\cdot Z_i
\\
=&Z_i^{-1}\cdot(T_i^{-1}\otimes T_i^{-1})(\Delta(u))\cdot Z_i,
\end{align*}
and hence \eqref{eq:T1} holds.
The proof of \eqref{eq:T2} is similar.
\end{proof}

Using Proposition \ref{prop:T} we see easily the following (see \cite[Lemma 2.8]{T}).
\begin{lemma}
\label{lem:TP}
We have
\begin{align*}
&\Delta(T_i(U^{\geqq0}))\subset U\otimes T_i(U^{\geqq0}),\quad
\Delta(T_i(U^{\leqq0}))\subset T_i(U^{\leqq0})\otimes U,
\\
&\Delta(T_i^{-1}(U^{\geqq0}))\subset T_i^{-1}(U^{\geqq0})\otimes U,\quad
\Delta(T_i^{-1}(U^{\leqq0}))\subset U\otimes T_i^{-1}(U^{\leqq0}).
\end{align*}
\end{lemma}

\begin{lemma}
\label{lem:p0}
We have
\begin{align*}
&U^+\cap T_i(U^{\geqq0})=U^+\cap T_i(U^+),\quad
U^-\cap T_i(U^{\leqq0})=U^-\cap T_i(U^-),
\\
&U^+\cap T_i^{-1}(U^{\geqq0})=U^+\cap T_i^{-1}(U^+),\quad
U^-\cap T_i^{-1}(U^{\leqq0})=U^-\cap T_i^{-1}(U^-).
\end{align*}
\end{lemma}
\begin{proof}
We only show the first formula since the proof of the others are similar.

Let
$u\in U^+\cap T_i(U^{\geqq0})$.
Let $V$ be an integrable $U$-module, and let $v\in V$.
For
$\lambda\in P^+$ we have
\begin{align*}
T_i^{-1}(u)(v\otimes v_\lambda)
=
(Z_i^{-1}\cdot(T_i^{-1}\otimes T_i^{-1})(\Delta(u)))(v\otimes v_\lambda)
\end{align*}
by \eqref{eq:T1}.
By Lemma \ref{lem:TP} we have 
\[
\Delta(u)\in u\otimes1+U^{\geqq0}\otimes((\bigoplus_{\gamma\in Q^+\setminus\{0\}}U^+_\gamma)\cap T_i(U^{\geqq0})),
\]
and hence
\[
(T_i^{-1}\otimes T_i^{-1})(\Delta(u))
\in T_i^{-1}(u)\otimes1+U\otimes U^0(\bigoplus_{\gamma\in Q^+\setminus\{0\}}U^+_\gamma).
\]
Therefore, we have
\[
T_i^{-1}(u)(v\otimes v_\lambda)
=Z_i^{-1}
(T_i^{-1}(u)v\otimes v_\lambda)
=T_i^{-1}(u)v\otimes v_\lambda.
\]
Write
\[
T_i^{-1}(u)=\sum_{h\in\Gh_\BZ}u_h k_h\qquad(u_\gamma\in U^+).
\]
Then we have
\begin{align*}
T_i^{-1}(u)(v\otimes v_\lambda)=&
\sum_h
u_h (k_h v\otimes k_h v_\lambda)
=\sum_h q^{\langle\lambda,h\rangle}
u_h(k_h v\otimes v_\lambda)
\\
=&\sum_h q^{\langle\lambda,h\rangle}
u_hk_h v\otimes v_\lambda,
\\
T_i^{-1}(u)v\otimes v_\lambda
=&\sum_h 
u_h k_h v\otimes v_\lambda,
\end{align*}
and hence
$\sum_{h\in\Gh_\BZ}(q^{\langle\lambda,h\rangle}-1)u_hV=\{0\}$
for any integrable $U$-module $V$.
By \cite[3.5.4]{Lbook} we obtain $\sum_{h\in\Gh_\BZ}(q^{\langle\lambda,h\rangle}-1)u_h=0$ for any $\lambda\in P^+$.
From this we see easily that 
$u_h=0$ for any $h\ne0$.
\end{proof}
By Lemma \ref{lem:TP} and \ref{lem:p0} we obtain
\begin{align}
\label{eq:TP1}
\Delta(U^+\cap T_i(U^+))\subset& U^{\geqq0}\otimes (U^+\cap T_i(U^+)),
\\
\label{eq:TP2}
\Delta(U^-\cap T_i(U^-))\subset& (U^-\cap T_i(U^-))\otimes U^{\leqq0},
\\
\label{eq:TP3}
\Delta(U^+\cap T_i^{-1}(U^+))\subset& (U^+\cap T_i^{-1}(U^+))U^0\otimes U^{\geqq0},
\\
\label{eq:TP4}
\Delta(U^-\cap T_i^{-1}(U^-))\subset& U^{\leqq0}\otimes (U^-\cap T_i^{-1}(U^-))U^0.
\end{align}

\section{Drinfeld pairing}
Set
\[
\tilde{U}^0=\bigoplus_{\gamma\in Q}\BF k_\gamma\subset U^0,
\qquad
\tilde{U}^{\geqq0}=\tilde{U}^0U^+,
\qquad
\tilde{U}^{\leqq0}=\tilde{U}^0U^-.
\]
The Drinfeld pairing is the bilinear form
\[
\tau:\tilde{U}^{\geqq0}\otimes\tilde{U}^{\leqq0}\to\BF
\]
characterized by the following properties:
\begin{align}
\label{eq:D1}
&
\tau(x,y_1y_2)=(\tau\otimes\tau)(\Delta(x),y_1\otimes y_2)
&(x\in \tilde{U}^{\geqq0}, y_1, y_2\in\tilde{U}^{\leqq0}),
\\
\label{eq:D2}
&
\tau(x_1x_2,y)=(\tau\otimes\tau)(x_2\otimes x_1,\Delta(y))
&(x_1, x_2\in \tilde{U}^{\geqq0}, y\in\tilde{U}^{\leqq0}),
\\
\label{eq:D3}
&
\tau(k_\gamma,k_\delta)=q^{-(\gamma,\delta)}
&(\gamma, \delta\in Q),
\\
\label{eq:D4}
&
\tau(e_i,f_j)=-\delta_{ij}(q_i-q_i^{-1})^{-1}
&(i, j\in I),
\\
\label{eq:D5}
&
\tau(e_i,k_\gamma)=\tau(k_\gamma,f_i)=0
&(i\in I, \gamma\in Q).
\end{align}
It satisfies the following properties:
\begin{align}
\label{eq:D6}
&
\tau(xk_\gamma,yk_\delta)=\tau(x,y)q^{-(\gamma,\delta)}
&(x\in U^+, y\in U^-,\gamma, \delta\in Q),
\\
\label{eq:D7}
&
\tau(U^+_\gamma, U^-_{-\delta})=\{0\}
&
(\gamma, \delta\in Q^+, \gamma\ne\delta),
\\
\label{eq:D8}
&
\tau|_{U^+_\gamma\times U^-_{-\gamma}}\;\text{is non-degenerate}
&(\gamma\in Q^+),
\\
\label{eq:D9}
&
\tau(Sx,Sy)=\tau(x,y)
&(x\in \tilde{U}^{\geqq0}, y\in\tilde{U}^{\leqq0}).
\end{align}
Moreover, for $x\in\tilde{U}^{\geqq0}$, $y\in \tilde{U}^{\leqq0}$ we have
\begin{align}
\label{eq:D10}
xy=&\sum_{(x)_2, (y)_2}
\tau(x_{(0)},y_{(0)})\tau(x_{(2)},Sy_{(2)})y_{(1)}x_{(1)},
\\
\label{eq:D11}
yx=&\sum_{(x)_2, (y)_2}
\tau(Sx_{(0)},y_{(0)})\tau(x_{(2)},y_{(2)})x_{(1)}y_{(1)}
\end{align}
(see \cite{T0}).

For the sake of completeness we include proofs of several well-known facts concerning $\tau$.
\begin{lemma}[see Proposition 38.1.6 of \cite{Lbook}]
\label{lem:pr}
We have
\begin{align*}
U^+\cap T_i(U^+)
=&
\{u\in U^+\mid\tau(u,U^-f_i)=\{0\}\},
\\
U^-\cap T_i(U^-)
=&
\{u\in U^-\mid\tau(U^+e_i,u)=\{0\}\},
\\
U^+\cap T_i^{-1}(U^+)
=&
\{u\in U^+\mid\tau(u,f_iU^-)=\{0\}\},
\\
U^-\cap T_i^{-1}(U^-)
=&
\{u\in U^-\mid\tau(e_iU^+,u)=\{0\}\}.
\end{align*}
\end{lemma}
\begin{proof}
We only show the first formula since the proof of the others are similar.

Assume $u\in U^+\cap T_i(U^+)$.
By
$U^+\cap T_i(U^+)\subset\bigoplus_{\gamma\in Q^+\cap s_iQ^+}U^+_\gamma$ and \eqref{eq:D7} we have $\tau(U^+\cap T_i(U^+),f_i)=0$.
Hence by \eqref{eq:TP1} we obtain
\[
\tau(u,U^-f_i)
=\sum_{(u)}\tau(u_{(0)},U^-)\tau(u_{(1)},f_i)=\{0\}.
\]

Assume $u\in U^+$ satisfies $\tau(u,U^-f_i)=\{0\}$.
We have only to show $T_i^{-1}(u)\in U^{\geqq0}$.
By Proposition \ref{prop:pos} it is sufficient to show that for any integrable $U$-module $V$ and any $\lambda\in P^+$ we have
\begin{equation}
\label{eq:claim}
T_i^{-1}(u)(V\otimes v_\lambda)\subset V\otimes v_\lambda.
\end{equation}

We first show
\begin{equation}
\label{eq:pr1}
\Delta(u)\subset \tilde{U}^{\geqq0}\otimes
\left(
\bigoplus_{\gamma\in Q^+\setminus\BZ_{>0}\alpha_i}U^+_\gamma
\right).
\end{equation}
For $r>0$ define $u_r\in U^+$ by
\[
\Delta(u)\in\sum_{r>0}u_rk_i^r\otimes e_i^r
+\tilde{U}^{\geqq0}\otimes
\left(
\bigoplus_{\gamma\in Q^+\setminus\BZ_{>0}\alpha_i}U^+_\gamma
\right).
\]
Then for 
$y\in U^-$, $m>0$ we have
\begin{align*}
0=&\tau(u, yf_i^m)=
\sum_{(u)}\tau(u_{(0)},y)\tau(u_{(1)},f_i^m)
=\tau(u_mk_i^m,y)\tau(e_i^m,f_i^m)
\\
=&\tau(u_m,y)\tau(e_i^m,f_i^m).
\end{align*}
By $\tau(e_i^m,f_i^m)\ne0$ we obtain $u_m=0$ for any $m>0$.
We have verified 
\eqref{eq:pr1}.

On the other hand by
$
U^+_\gamma f_i^m\subset
\sum_{r=0}^mf_i^rU^0U^+_{\gamma-r\alpha_i}
$
we have
\begin{equation}
\label{eq:pr2}
m\in\BZ_{\geqq0},\;\gamma\in Q^+\setminus\BZ_{\geqq0}\alpha_i
\;
\Longrightarrow
\;
U^+_\gamma f_i^mv_\lambda=\{0\}.
\end{equation}

Now we can show \eqref{eq:claim}.
By \eqref{eq:T1} we have
\begin{align*}
T_i^{-1}(u)(V\otimes v_\lambda)
=&
Z_i^{-1}(T_i^{-1}\otimes T_i^{-1})\Delta(u)(V\otimes T_iv_\lambda)
\\
=&
Z_i^{-1}(T_i^{-1}\otimes T_i^{-1})\Delta(u)(V\otimes 
f_i^{\langle\lambda,h_i\rangle}v_\lambda).
\end{align*}
By \eqref{eq:pr1}, \eqref{eq:pr2} we have 
\[
\Delta(u)(V\otimes f_i^{\langle\lambda,h_i\rangle}v_\lambda)
=(u\otimes1)(V\otimes f_i^{\langle\lambda,h_i\rangle}v_\lambda)
\subset V\otimes f_i^{\langle\lambda,h_i\rangle}v_\lambda,
\]
and hence 
\begin{align*}
T_i^{-1}(u)(V\otimes v_\lambda)
\subset
Z_i^{-1}(V\otimes 
T_i^{-1}f_i^{\langle\lambda,h_i\rangle}v_\lambda)
=
Z_i^{-1}(V\otimes v_\lambda)
=V\otimes v_\lambda.
\end{align*}
\end{proof}
\begin{lemma}[see Lemma 38.1.2 of \cite{Lbook}]
\label{lem:ten}
The multiplication of $U$ induces isomorphisms
\begin{align*}
U^+
\cong&
\BF[e_i]\otimes(U^+\cap T_i^{\pm1}(U^+))
\cong(U^+\cap T_i^{\pm1}(U^+))\otimes\BF[e_i],
\\
U^-
\cong&
\BF[f_i]\otimes(U^-\cap T_i^{\pm1}(U^-))
\cong(U^-\cap T_i^{\pm1}(U^-))\otimes\BF[f_i].
\end{align*}
\end{lemma}
\begin{proof}
We only show 
$U^+\cong\BF[e_i]\otimes(U^+\cap T_i(U^+))$ since other formulas are proved similarly.
The injectivity of $\BF[e_i]\otimes(U^+\cap T_i(U^+))\to U^+$ follows from
$T_i^{-1}(\BF[e_i])\otimes T_i^{-1}(U^+\cap T_i(U^+))\subset U^{\leqq0}\otimes U^+\cong U$.
Hence it is sufficient to show that for any $\gamma\in Q$ we have
\[
\dim U^+_\gamma=\sum_{r\geqq0}\dim (U^+_{\gamma-r\alpha_i}\cap T_i(U^+)).
\]
For $\delta\in Q$ we have
$\dim (U^-_{-\delta}\cap U^-f_i)=\dim U^-_{-(\delta-\alpha_i)}=\dim U^+_{\delta-\alpha_i}$,
and hence $\dim (U^+_\delta\cap T_i(U^+))=\dim U^+_\delta-\dim U^+_{\delta-\alpha_i}$ by Lemma \ref{lem:pr}, \eqref{eq:D7}, \eqref{eq:D8}.
It follows that 
\[
\sum_{r\geqq0}\dim (U_{\gamma-r\alpha_i}\cap T_i(U^+))
=\sum_{r\geqq0}(\dim U^+_{\gamma-r\alpha_i}-\dim U^+_{\gamma-(r+1)\alpha_i})=\dim U^+_\gamma
\]
since $\dim U^+_{\gamma-r\alpha_i}=0$ for $r\gg0$.
\end{proof}
\begin{lemma}[see Proposition 38.2.3 of \cite{Lbook}]
\label{lem:sep}
\begin{itemize}
\item[(i)]
For $x\in U^+\cap T_i(U^+)$, $y\in U^-\cap T_i(U^-)$, $m, n\in\BZ_{\geqq0}$
we have
\[
\tau(xe_i^m,yf_i^n)=\tau(x,y)\tau(e_i^m,f_i^n).
\]
\item[(ii)]
For $x\in U^+\cap T_i^{-1}(U^+)$, $y\in U^-\cap T_i^{-1}(U^-)$, $m, n\in\BZ_{\geqq0}$
we have
\[
\tau(e_i^mx,f_i^ny)=\tau(x,y)\tau(e_i^m,f_i^n).
\]
\end{itemize}
\end{lemma}
\begin{proof}
We only show (i) since (ii) is proved similarly.
For $x\in U^+\cap T_i(U^+)$, $y\in U^-\cap T_i(U^-)$, $m, n\in\BZ_{\geqq0}$
we have
\begin{align*}
\tau(xe_i^m,yf_i^n)=&
(\tau\otimes\tau)(e_i^m\otimes x,\Delta(yf_i^n))
\\
=&(\tau\otimes\tau)(e_i^m\otimes x,\sum_{(y)}y_{(0)}f_i^n\otimes y_{(1)}k_i^{-n}))
\\
=&
\sum_{(y)}
\tau(e_i^m,y_{(0)}f_i^n)\tau(x,y_{(1)}k_i^{-n})
\\
=&
\sum_{(y)}
(\tau\otimes\tau)(\Delta(e_i^m),y_{(0)}\otimes f_i^n)\tau(x,y_{(1)})
\\
=&
\sum_{(y)}
\tau(k_i^m,y_{(0)})
\tau(e_i^m,f_i^n)
\tau(x,y_{(1)})
=\tau(e_i^m,f_i^n)
\tau(x,y).
\end{align*}
Here, the second identity follows from Lemma \ref{lem:pr}, and the fifth identity is a consequence of 
\eqref{eq:TP2}
and Lemma \ref{lem:pr}.
The statement (i) is proved.
\end{proof}

\section{Invariance}
\subsection{Main result}
The purpose of this note is to give two simple proofs of the following fact.
\begin{theorem}[see Proposition 38.2.1 of \cite{Lbook}]
\label{theorem}
For $x\in U^+\cap T_i(U^+)$, $y\in U^-\cap T_i(U^-)$ we have
\[
\tau(T_i^{-1}(x), T_i^{-1}(y))=\tau(x, y).
\]
\end{theorem}
\subsection{The first proof}

By the triangular decomposition $U\cong U^-\otimes U^0\otimes U^+$ we have
\[
U=\{(U^-\cap\Ker(\varepsilon))U+U(U^+\cap\Ker(\varepsilon))\}
\oplus U^0.
\]
We define a linear map
\[
p:U\to U^0
\]
as the projection with respect to this direct sum decomposition.
The following fact is crucial.
\begin{proposition}[see Proposition 19.3.7 of \cite{Lbook}]
\label{prop:DS}
Let $\gamma\in Q^+$, and let $x\in U^+_\gamma$, $y\in U^-_{-\gamma}$.
Assume 
\[
\Delta(x)\in x\otimes1+
\sum_{\delta\in X}
U^{\geqq0}\otimes U^+_\delta
\]
for $X\subset Q^+\setminus\{0\}$.
Then we have
\[
p(xy)\in k_{-\gamma}\left(\tau(x,y)+\sum_{\delta\in X}\BF k_{2\delta}\right).
\]
\end{proposition}
\begin{proof}
Writing
\begin{align*}
\Delta(x)=&
\sum_r x'_rk_{\delta_r}\otimes x_r
\qquad
(\delta_r\in Q^+, x_r\in U^+_{\delta_r}, x'_r\in U^+_{\gamma-\delta_r}),
\\
\Delta(y)=&
\sum_s y_s\otimes k_{-\gamma_s}y'_s
\qquad
(\gamma_s\in Q^+, y_s\in U^-_{-\gamma_s}, y'_s\in U^-_{-(\gamma-\gamma_s)})
\end{align*}
we have
\begin{align*}
\Delta_2(x)\in&
\sum_r x'_rk_{\delta_r}\otimes k_{\delta_r}\otimes x_r
+U^{\geqq0}\otimes U^0(U^+\cap\Ker(\varepsilon))\otimes U^+,
\\
\Delta_2(y)\in&
\sum_s y_s\otimes k_{-\gamma_s}\otimes k_{-\gamma_s}y'_s
+
U^{-}\otimes (U^-\cap\Ker(\varepsilon))U^0\otimes U^{\leqq0}.
\end{align*}
Hence by \eqref{eq:D6}, \eqref{eq:D7}, \eqref{eq:D10} we have
\begin{align*}
p(xy)=&
\sum_{\delta_r+\gamma_s=\gamma}
\tau(x'_rk_{\delta_r},y_s)
\tau(x_r,S(k_{-\gamma_s}y'_s))k_{\delta_r-\gamma_s}
\\
=&
k_{-\gamma}\left(
\sum_{\delta_r+\gamma_s=\gamma}
\tau(x'_r,y_s)
\tau(x_r,S(y'_s))k_{2\delta_r}
\right),
\end{align*}
from which we easily obtain our desired result.
\end{proof}

Now let us give our first proof of Theorem \ref{theorem}.
We may assume 
$x\in U^+_\gamma\cap T_i(U^+)$, $y\in U^-_{-\gamma}\cap T_i(U^-)$ for $\gamma\in Q^+$.
By Proposition \ref{prop:DS} it is sufficient to show
\[
p(T_i^{-1}(x)T_i^{-1}(y))
\in 
k_{-s_i(\gamma)}
\left(
\tau(x,y)+\sum_{\delta\in Q^+\setminus\{0\}}\BF k_{2\delta}
\right).
\]
By \eqref{eq:TP1}, \eqref{eq:TP2}
we can write
\begin{align*}
\Delta(x)=
\sum_r x'_rk_{\delta_r}\otimes x_r,
\qquad
\Delta(y)=&
\sum_s y_s\otimes k_{-\gamma_s}y'_s,
\end{align*}
where
$\delta_r, \gamma_s\in Q^+\cap s_iQ^+$, 
$x_r\in U^+_{\delta_r}\cap T_i(U^+)$, $x'_r\in U^+_{\gamma-\delta_r}$, 
$y_s\in U^-_{-\gamma_s}\cap T_i(U^-)$, $y'_s\in U^-_{-(\gamma-\gamma_s)}$.
Furthermore, by \eqref{eq:TP1}, \eqref{eq:TP2} and Lemma \ref{lem:ten}
we can write
\begin{align*}
\Delta(x_r)\in&
\sum_{m\geqq0} 
e_i^{(m)}k_{\delta_{r}-m\alpha_i}\otimes x_{rm}
+U^{\geqq0}(U^+\cap T_i(U^+)\cap\Ker(\varepsilon))\otimes U^+,
\\
\Delta(y_s)\in&
\sum_{n\geqq0} 
y_{sn}
\otimes
k_{-(\gamma_s-n\alpha_i)}
f_i^{(n)}
+
U^-\otimes
(U^-\cap T_i(U^-)\cap\Ker(\varepsilon))U^{\leqq0},
\end{align*}
where $x_{rm}\in U^+_{\delta_{r}-m\alpha_i}\cap T_i(U^+)$, 
$y_{sn}\in U^-_{-(\gamma_s-n\alpha_i)}\cap T_i(U^-)$.
Then we have
\begin{multline*}
\Delta_2(x)
-\sum_{r,m}
x'_rk_{\delta_r}
\otimes
e_i^{(m)}k_{\delta_{r}-m\alpha_i}
\otimes 
x_{rm}
\\
\in
U^{\geqq0}\otimes U^{\geqq0}(U^+\cap T_i(U^+)\cap\Ker(\varepsilon))\otimes U^+,
\end{multline*}
\begin{multline*}
\Delta_2(y)
-\sum_{s,n}
y_{sn}
\otimes
k_{-(\gamma_s-n\alpha_i)}
f_i^{(n)}
\otimes 
k_{-\gamma_s}y'_s
\\
\in
U^-\otimes
(U^-\cap T_i(U^-)\cap\Ker(\varepsilon))U^{\leqq0}
\otimes U^{\leqq0}.
\end{multline*}
Hence by \eqref{eq:D6}, \eqref{eq:D7}, \eqref{eq:D10} we obtain
\begin{multline}
\label{eq:xy}
xy-
\sum_{\gamma_s+\delta_r=\gamma-m\alpha_i}
\tau(x'_r,y_{sm})
\tau(x_{rm},S(y'_s))
k_{-(\gamma_s-m\alpha_i)}
f_i^{(m)}
e_i^{(m)}k_{\delta_{r}-m\alpha_i}
\\
\in
(U^-\cap T_i(U^-)\cap\Ker(\varepsilon))U
+
U(U^+\cap T_i(U^+)\cap\Ker(\varepsilon)).
\end{multline}
In particular, we have
\[
p(xy)=
\sum_{\gamma_s+\delta_r=\gamma}
\tau(x'_r,y_{s0})
\tau(x_{r0},S(y'_s))
k_{-\gamma+2\delta_r},
\]
and hence 
\begin{equation}
\label{eq:txy}
\tau(x,y)=
\sum_{\gamma_s=\gamma, \delta_r=0}
\tau(x'_r,y_{s0})
\tau(x_{r0},S(y'_s))
\end{equation}
by Proposition \ref{prop:DS}.
Next we apply $T_i^{-1}$ to \eqref{eq:xy}.
We can easily check that 
\[
T_i^{-1}(f_i^{(m)}e_i^{(m)})=e_i^{(m)}f_i^{(m)}
\in 
\begin{bmatrix}
k_i
\\
m
\end{bmatrix}+U(U^+\cap\Ker(\varepsilon))+(U^-\cap\Ker(\varepsilon))U,
\]
where
\[
\begin{bmatrix}
k_i
\\
m
\end{bmatrix}
=
\prod_{r=1}^m\frac{q_i^{-(r-1)}k_i-q_i^{r-1}k_i^{-1}}{q_i^r-q_i^{-r}}.
\]
It follows that 
\begin{multline*}
T_i^{-1}(xy)
-
\sum_{\gamma_s+\delta_r=\gamma-m\alpha_i}
\tau(x'_r,y_{sm})
\tau(x_{rm},S(y'_s))
\begin{bmatrix}
k_i
\\
m
\end{bmatrix}
k_{s_i(\delta_{r}-\gamma_s)}
\\
\in
U(U^+\cap\Ker(\varepsilon))+(U^-\cap\Ker(\varepsilon))U,
\end{multline*}
and hence
\[
p(T_i^{-1}(xy))=
\sum_{\gamma_s+\delta_r=\gamma-m\alpha_i}
\tau(x'_r,y_{sm})
\tau(x_{rm},S(y'_s))
\begin{bmatrix}
k_i
\\
m
\end{bmatrix}
k_{s_i(\delta_{r}-\gamma_s)}.
\]
Note
\[
\begin{bmatrix}
k_i
\\
m
\end{bmatrix}
\in
k_{-m\alpha_i}
\left(
\BF^\times
+\sum_{n>0}\BF k_{2n\alpha_i}
\right).
\]
If 
$\gamma_s+\delta_r=\gamma-m\alpha_i$, 
then we have 
\[
s_i(\delta_r-\gamma_s)-m\alpha_i=
-s_i\gamma+2s_i(\delta_r-m\alpha_i).
\]
Recall that $x_{rm}\in U^+_{\delta_r-m\alpha_i}\cap T_i(U^+)$.
Hence if $x_{rm}\ne0$, then $s_i(\delta_r-m\alpha_i)\in Q^+$.
Moreover, by $\delta_r\in Q^+\cap s_iQ^+$, $\delta_r-m\alpha_i=0$ happens only if $\delta_r=0$ and $m=0$.
It follows that 
\begin{align*}
p(T_i^{-1}(xy))\in
&k_{-s_i\gamma}
\left(
\sum_{\delta_r=0,\gamma_s=\gamma}
\tau(x'_r,y_{s0})
\tau(x_{r0},S(y'_s))+\sum_{\delta\in Q^+\setminus\{0\}}\BF k_{2\delta}
\right)
\\
=&k_{-s_i\gamma}
\left(
\tau(x,y)+\sum_{\delta\in Q^+\setminus\{0\}}\BF k_{2\delta}
\right)
\end{align*}
by \eqref{eq:txy}.
The proof is complete.

\subsection{The second proof}
For each $\gamma\in Q^+$ we denote by $\Theta_\gamma\in U^+_\gamma\otimes U^-_{-\gamma}$ the canonical element of the non-degenerate bilinear form $\tau|_{U^+_\gamma\times U^-_{-\gamma}}$.
Namely, for bases $\{x_j\}$, $\{y_j\}$ of $U^+_\gamma$, $U^-_{-\gamma}$ respectively such that $\tau(x_j,y_k)=\delta_{jk}$ we set
$\Theta_\gamma=\sum_jx_j\otimes y_j$.
We regard the infinite sum
\begin{equation}
\Theta=\sum_{\gamma\in Q^+}\Theta_\gamma
\end{equation}
as an operator on the tensor product of two integrable $U$-modules.
For $u\in U$ we set
\[
\Delta'(u)=P(\Delta(u)),
\]
where $P(u_1\otimes u_2)=u_2\otimes u_1$.
The following fact is crucial.
\begin{proposition}[see Theorem 4.1.2 of \cite{Lbook}]
\label{prop:R}
We have
\begin{equation}
\label{eq:R}
\Delta'(u)\cdot\Theta=\Theta\cdot(\Phi(\Delta(u)))
\qquad(u\in U).
\end{equation}
Moreover, the family $\Theta_\gamma\in U^+_\gamma\otimes U^-_{-\gamma}$ $ (\gamma\in Q^+)$ is uniquely determined by the equation \eqref{eq:R}.
\end{proposition}

Let us give our second proof of Theorem \ref{theorem}.

Define a bilinear form 
\[
\tilde{\tau}:U^+\times U^-\to\BF
\]
by
\[
\tilde{\tau}(xe_i^m,yf_i^n)=
\tau(T_i^{-1}(x),T_i^{-1}(y))\tau(e_i^m,f_i^n)
\]
for $x\in U^+\cap T_i(U^+)$, $y\in U^-\cap T_i(U^-)$, $m,n\in\BZ_{\geqq0}$ (see Lemma \ref{lem:ten}).
Then it is sufficient to show $\tau|_{U^+\times U^-}=\tilde{\tau}$ in view of Lemma \ref{lem:sep}.
For $\gamma\in Q^+$ let $\tilde{\Theta}_\gamma$ be the canonical element of $\tilde{\tau}|_{U^+_\gamma\times U^-_{-\gamma}}$, and set
$\tilde{\Theta}=\sum_{\gamma\in Q^+}\tilde{\Theta}_\gamma$.
Since $\tau|_{U^+\times U^-}$ and $\tilde{\tau}$ are uniquely determined by $\Theta$ and $\tilde{\Theta}$  respectively, it is sufficient to show $\Theta=\tilde{\Theta}$.
Moreover, by the uniqueness in Proposition \ref{prop:R}
this is equivalent to 
\begin{equation}
\label{eq:RR}
\Delta'(u)\cdot \tilde{\Theta}=\tilde{\Theta}\cdot\Phi(\Delta(u))\qquad
(u\in U).
\end{equation}

For $\gamma\in Q^+\cap s_i(Q^+)$ let $\Theta'_\gamma$ and $\Theta''_\gamma$ be the canonical elements of 
$\tau|_{(U^+_\gamma\cap T_i(U^+))\times (U^-_{-\gamma}\cap T_i(U^-))}$
and
$\tau|_{(U^+_\gamma\cap T_i^{-1}(U^+))\times (U^-_{-\gamma}\cap T_i^{-1}(U^-))}$
respectively, and set
$\Theta'=\sum_{\gamma\in Q^+\cap s_i(Q^+)}\Theta'_\gamma$ and
$\Theta''=\sum_{\gamma\in Q^+\cap s_i(Q^+)}\Theta''_\gamma$.
By Lemma \ref{lem:sep} and the formula 
\[
\tau(e_i^m,f_i^n)=\delta_{mn}
\frac{q_i^{n(n-1)/2}}{(q_i^{-1}-q_i)^n}[n]!_{q_i}
\]
we have 
\begin{equation}
\label{eq:rel}
\Theta=\Theta'\cdot R_i=R_i\cdot\Theta'',
\qquad
\tilde{\Theta}=(T_i\otimes T_i)(\Theta'')\cdot R_i.
\end{equation}
It follows that 
\begin{align*}
\Delta'(u)\cdot\tilde{\Theta}
=&
\Delta'(u)\cdot(T_i\otimes T_i)(\Theta'')\cdot R_i
\\
=&
(T_i\otimes T_i)((T_i^{-1}\otimes T_i^{-1})(\Delta'(u))\cdot\Theta'')\cdot R_i
\\
=&
(T_i\otimes T_i)(
R_i^{-1}\cdot
\Delta'(T_i^{-1}(u))
\cdot
R_i\Theta'')\cdot R_i
\\
=&
(T_i\otimes T_i)(
R_i^{-1}\cdot
\Delta'(T_i^{-1}(u))
\cdot
\Theta)\cdot R_i
\\
=&
(T_i\otimes T_i)(
R_i^{-1}\Theta\cdot
\Phi(\Delta(T_i^{-1}(u)))
)\cdot R_i
\\
=&
(T_i\otimes T_i)(
\Theta''\cdot
\Phi(\Delta(T_i^{-1}(u)))
)\cdot R_i
\\
=&
\tilde{\Theta}R_i^{-1}\cdot
(T_i\otimes T_i)(
\Phi(\Delta(T_i^{-1}(u)))
)\cdot R_i
\\
=&
\tilde{\Theta}
R_i^{-1}\cdot
\Phi(
(T_i\otimes T_i)(
\Delta(T_i^{-1}(u))
)
)
\cdot R_i
\\
=&
\tilde{\Theta}\cdot
\Phi(
\Phi^{-1}(R_i)^{-1}\cdot
(T_i\otimes T_i)(
\Delta(T_i^{-1}(u))
)
\cdot
\Phi^{-1}(R_i))
\\
=&\tilde{\Theta}\cdot\Phi(\Delta(u))
\end{align*}
by \eqref{eq:Phi}, \eqref{eq:T1}, \eqref{eq:T2}, \eqref{eq:rel}.
We have proved \eqref{eq:RR}, and hence our second proof of Theorem \ref{theorem} is complete.

\bibliographystyle{unsrt}

\end{document}